\documentclass[11pt, a4paper]{amsart}
\usepackage{amssymb, mathtools, enumerate, hyperref}
\mathtoolsset{showonlyrefs=true}

\DeclareMathOperator*{\Ran}{Ran}
\DeclareMathOperator*{\Ker}{Ker}
\DeclareMathOperator*{\Fix}{Fix}
\DeclareMathOperator*{\dist}{dist}

\newcommand{\dd}{\mathrm{d}}
\newcommand{\ep}{\varepsilon}

\renewcommand{\L}{L_n^{(1)}}
\newcommand{\TT}{\mathbb{T}}

\newcommand{\CC}{\mathbb{C}}
\newcommand{\ZZ}{\mathbb{Z}}

\newcommand{\ga}{\alpha}
\newcommand{\inv}{^{-1}}

\newtheorem{thm}{Theorem}[section]
\newtheorem{prp}[thm]{Proposition}
\newtheorem{lem}[thm]{Lemma}

\theoremstyle{definition}
\newtheorem{rem}[thm]{Remark}
\newtheorem{rems}[thm]{Remarks}

\numberwithin{equation}{section} 

\begin{document}

\title[The asymptotic behaviour of the Cesàro operator]{The asymptotic behaviour of\\ the Cesàro operator}

\author[A.K.J. Pritchard and D. Seifert]{Andrew K.~J.~Pritchard and David Seifert}
\address{School of Mathematics, Statistics and Physics, Newcastle University, Herschel Building, Newcastle upon Tyne, NE1 7RU, United Kingdom}
\email[A.K.J.\ Pritchard]{a.k.j.pritchard2@ncl.ac.uk}
\email[D.\ Seifert]{david.seifert@ncl.ac.uk}

\begin{abstract}
We study the asymptotic behaviour of orbits $(T^nx)_{n\ge0}$ of the classical Cesàro operator $T$ for sequences $x$ in the Banach space $c$ of convergent sequences. We give new non-probabilistic proofs, based on the Katznelson--Tzafriri theorem and one of its quantified variants, of results which characterise the set of sequences $x\in c$ that lead to convergent orbits and, for sequences satisfying a simple additional condition, we provide a rate of convergence. These results are then shown, again by operator-theoretic techniques, to be optimal in different ways. Finally, we study the asymptotic behaviour of the Cesàro operator defined on spaces of continuous functions, establishing new and improved results in this setting, too.
\end{abstract}



\maketitle

\section{Introduction}\label{sec:int}

Let $c$ be the complex Banach space of complex convergent sequences endowed with the supremum norm.
We consider the classical \emph{Cesàro operator} $T\colon c\to c$ defined by $Tx=(\phi_k(x))_{k\ge0}$, where
\begin{equation}\label{eq:phi}
\phi_k(x)=\frac{1}{k+1}\sum_{j=0}^kx_j,\qquad k\ge0,
\end{equation}
for $x=(x_k)_{k\ge0}\in c$. It well known that $T$ is then a well-defined bounded linear operator on $c$, of norm $\|T\|=1$.
Our primary interest is in the convergence properties of the orbits $(T^nx)_{n\ge0}$ of the Cesàro operator for different sequences $x\in c$.  It was shown by Galaz Fontes and Solís in~\cite[Thm.~1]{GalSol08} that the orbit $(T^nx)_{n\ge0}$ converges in norm to a limit as $n\to\infty$ if and only if the sequence $x=(x_k)_{k\ge0}\in c$ satisfies $x_0=\lim_{k\to\infty}x_k$. Shortly afterwards, Adell and Lekuona showed in~\cite[Thm.~1.3]{AdeLek10} that for sequences satisfying a certain additional integral condition the convergence occurs at the rate $n^{-1/2}$ as $n\to\infty$. In both cases, the proof is based on the representation of powers of $T$ as so-called \emph{moment-difference operators} corresponding to a certain measure, and the rate of convergence in~\cite[Thm.~1.3]{AdeLek10} is obtained by means of a probabilistic argument based on the Berry--Esseen theorem, a quantified version of the central limit theorem. In two of our main results,  Theorems~\ref{thm:conv} and~\ref{thm:rate}, we give new operator-theoretic proofs of these results based on the famous Katznelson--Tzafriri theorem in operator theory and a result on generalised Laguerre polynomials. This new approach leads to a simpler additional condition for the convergence to occur at the rate $n^{-1/2}$ as $n\to\infty$, namely that the series $\sum_{k=1}^\infty\frac{x_k-x_0}{k}$ converges, and we show that our condition is implied by the integral condition in~\cite[Thm.~1.3]{AdeLek10}. In addition, in Proposition~\ref{prp:opt} we establish that our quantified convergence result is optimal both in the sense that the rate $n^{-1/2}$ as $n\to\infty$ cannot be improved and that this rate cannot hold for all convergent orbits. On the other hand, Proposition~\ref{prp:rapid} shows that the set of sequences that lead to convergent orbits contains a dense subspace for which the convergence is in fact more rapid than any polynomial rate.

One important benefit of our functional-analytic approach is that it carries over with only minor modifications to natural `continuous' analogues of the Cesàro operator. Thus in Section~\ref{sec:cont} we outline how our arguments can be adapted to both the Cesàro operator on $C[0,1]$ and to the Cesàro operator on suitable spaces of functions on the half-line $[0,\infty)$. In Theorem~\ref{thm:interval} we give the first result in the literature describing a set of functions $f\in C[0,1]$ for which the orbit $(T^nf)_{n\ge0}$ of the Cesàro operator  $T\colon C[0,1]\to C[0,1]$  converges to a limit at the rate $\smash{n^{-1/2}}$ as $n\to\infty$. For the half-line analogue of the Cesàro operator  a result of this type has previously been obtained in~\cite[Thm~1.1]{AdeLek10}, but in Theorem~\ref{thm:half} we prove it for a class including functions not previously covered.

Our notation and terminology are standard throughout. In particular, given a complex Banach space $X$ and a bounded linear operator $T\colon X\to X$, we denote the spectrum of $T$ by $\sigma(T)$, and we write $\Fix T$ for the eigenspace $\Ker (I-T)$ of vectors that are fixed by $T$. We denote the dual operator of $T$ by $T^*\colon X^*\to X^*$, where $X^*$ is the dual space of $X$. Furthermore, if $Y$ is a subset of the dual $X^*$ of a Banach space $X$, we let $Y_\circ=\{x\in X:\phi(x)=0\mbox{ for all }\phi\in Y\}$ denote the annihilator of $Y$ in $X$.

\section{The convergence result}

 Our first main result, Theorem~\ref{thm:conv} below, describes the set of sequences $x\in c$ for which the orbit $(T^nx)_{n\ge0}$ of the Cesàro operator $T\colon c\to c$  converges to a limit as $n\to\infty$ either in norm or, equivalently as it turns out, with respect to the weak topology. The main equivalence between (i) and (iii) was first proved in~\cite[Thm.~1]{GalSol08}, and the authors also described the limit when it exists. The proof given in~\cite{GalSol08} goes through Hardy's interpretation \cite[\S11.12]{Har49} of powers of $T$ as the moment-difference operators corresponding to a certain measure on $[0,1]$, and it relies on a lemma involving elementary calculus and Stirling's formula. Our proof is shorter and purely functional-analytic. It is based on the well-known Katznelson--Tzafriri theorem, which we state for the reader's convenience. Here $\TT=\{z\in\CC:|z|=1\}$ denotes the unit circle in the complex plane, and a bounded linear operator $T\colon X\to X$ on a Banach space $X$ is said to be \emph{power-bounded} if $\sup_{n\ge0}\|T^n\|<\infty$.

\begin{thm}[{\cite[Thm.~1]{KatTza86}}]\label{thm:KT}
Let $X$ be a complex Banach space, and let $T\colon X\to X$ be a power-bounded linear operator. Then $\|T^n(I-T)\|\to0$ as $n\to\infty$ if and only if $\sigma(T)\cap\TT\subseteq\{1\}$.
\end{thm}

One advantage of having a proof based on Theorem~\ref{thm:KT} is that there are sharp quantified versions of the Katznelson--Tzafiriri theorem.
 We shall use such a result in Section~\ref{sec:rates} below to give a precise description of the \emph{rate} of convergence one can expect in Theorem~\ref{thm:conv},  based on an auxiliary fact about generalised Laguerre polynomials and on  operator-theoretic ideas. For a recent general survey of results relating to the Katznelson--Tzafriri theorem we refer the reader to~\cite{BatSei22}. 

For $k\in\ZZ_+\cup\{\infty\}$ we define the bounded linear functionals $\pi_k\in c^*$  by $\pi_k(x)=x_k$ when $k\in\ZZ_+$ and by $\pi_\infty(x)= \lim_{k\to\infty}x_k$  for all $x=(x_k)_{k\ge0}\in c$.  In addition, we write $P\colon c\to c$ for the operator defined by $Px=\pi_0(x)e_\infty$ for $x\in c$, where $e_\infty=(1,1,1,\dots)$.  Our main convergence result may be stated as follows.

\begin{thm}\label{thm:conv}
Let  $T\colon c\to c$ be the Cesàro operator, and let $x=(x_k)_{k\ge0}\in c$. The following are equivalent:
\begin{enumerate}[(i)]
\item[\textup{(i)}] $x_0=\lim_{k\to\infty}x_k$;
\item[\textup{(ii)}] $(T^nx)_{n\ge0}$ converges weakly as $n\to\infty$;
\item[\textup{(iii)}] $(T^nx)_{n\ge0}$ converges in norm as $n\to\infty$.
\end{enumerate}
Furthermore, if these conditions hold then 
\begin{equation}\label{eq:limit}
\|T^nx-Px\|_\infty\to0,\qquad n\to\infty.
\end{equation}
\end{thm}

\begin{proof}
Let $x\in c$. It is clear that (iii) implies (ii). Suppose then that~(ii) holds and that $y\in c$ is such that $T^nx\to y$ weakly as $n\to\infty$. Then the sequence $(T^{n+1}x)_{n\ge0}$ converges weakly to both $Ty$ and to $y$  as $n\to\infty$. Since the weak topology is Hausdorff, we deduce that $y\in\Fix T=\langle e_\infty\rangle$. Thus $y$ is a constant sequence and in particular $\pi_\infty(y)=\pi_0(y)$. Now $\Fix T^*=\langle \pi_0,\pi_\infty\rangle$, for instance by \cite[Thm.~5]{Lei72}, and in particular $\pi_0(x)=\pi_0(T^nx)$ and $\pi_\infty(x)=\pi_\infty(T^nx)$ for all $n\ge0$.
By weak convergence, letting $n\to\infty$ gives $\pi_0(x)=\pi_0(y)=\pi_\infty(y)=\pi_\infty(x)$, so $x\in \Ker(\pi_0-\pi_\infty)$ and hence (ii) implies (i). 
In order to prove that (i) implies~(iii) we need to show that $\Ker(\pi_0-\pi_\infty)\subseteq X$, where
\begin{equation}\label{eq:X}
X=\big\{x\in c:(T^nx)_{n\ge0} \mbox{ converges in norm as $n\to\infty$}\big\}.
\end{equation}
 We begin by establishing the inclusion $\Ker(\pi_0-\pi_\infty)\subseteq Y\oplus\langle e_\infty\rangle$, where $Y$ denotes the norm-closure of $\Ran(I-T)$ in $c$.  Indeed, given $x\in \Ker(\pi_0-\pi_\infty)$ let $y=x-Px$. Then certainly $\pi_0(y)=0$, and moreover $\pi_\infty(y)=\pi_\infty(x)-\pi_0(x)=0$. Since $\Fix T^*=\langle \pi_0,\pi_\infty\rangle$ it follows from the Hahn--Banach theorem that $y\in Y$, and hence $x\in Y\oplus\langle e_\infty\rangle$ as required. Next we observe that the Cesàro operator, being a contraction, is  power-bounded. Moreover $\sigma(T)=\{z\in\CC:|z-\frac12|\le\frac12\}$, and hence $\sigma(T)\cap\TT=\{1\}$; see for instance~\cite{Lei72,Rea85, Wen75}. Thus  $\|T^n(I-T)\|\to0$ as $n\to\infty$ by Theorem~\ref{thm:KT}. It follows that $\Ran(I-T)\subseteq X$, and a simple approximation argument shows that $Y\subseteq X$, too. Since $e_\infty\in\Fix T$ we have $\langle e_\infty\rangle\subseteq X$, so $\Ker(\pi_0-\pi_\infty)\subseteq Y\oplus\langle e_\infty\rangle\subseteq X$. This completes the proof of the equivalence of (i), (ii) and (iii). From the proof of the implication (ii)$\implies$(i) it is clear that the limit, when it exists, must be as described in~\eqref{eq:limit}.
\end{proof}

\begin{rems}\phantomsection\label{rem1}
\begin{enumerate}[(a)]
\item The proof in fact establishes  that $X= Y\oplus\langle e_\infty\rangle,$ where $X$ is as defined in~\eqref{eq:X} and $Y$ denotes the norm-closure of $\Ran(I-T)$. This fact will be useful in Section~\ref{sec:rates} below. 
\item A similar argument to that used in the above proof shows that if $S\colon c^*\to c^*$ is the dual of the Cesàro operator then, given any $\phi\in c^*$, 
$$\bigg\|S^n\phi-\bigg(\sum_{k=0}^\infty \phi(e_k)\bigg)\pi_0-\bigg(\phi(e_\infty)-\sum_{k=0}^\infty \phi(e_k)\bigg)\pi_\infty\bigg\|\to0$$
as $n\to\infty$.
Indeed,  let $Q\colon c^*\to c^*$ be the map given by $Q\phi=\xi_0(\phi) \pi_0+\xi_\infty(\phi)\pi_\infty$, where $\xi_0,\xi_\infty\in c^{**}$ are defined by
$\xi_0(\phi)=\sum_{k=0}^\infty \phi(e_k)$
and $ \xi_\infty(\phi)=\phi(e_\infty)-\xi_0(\phi)$, respectively, for all $\phi\in c^*$. Then $Q$ is bounded, $Q^2=Q$ and $\Ran Q=\Fix S=\langle \pi_0,\pi_\infty\rangle$. Furthermore, 
$\Ker Q=\Ker\xi_0\cap\Ker\xi_\infty=(\Fix S^*)_\circ$.  
Thus $Q$ is the projection onto $\langle\pi_0,\pi_\infty\rangle$ along the closure of $\Ran(I-S)$. Since $\|S\|=\|T\|=1$ and $\sigma(S)=\sigma(T)$,  Theorem~\ref{thm:KT} yields $S^n\to Q$ as $n\to\infty$, in the strong operator topology. This result is an analogue of~\cite[Thm.~2]{GalSol08} for the Cesàro operator on $c$.
\item Applying part~(b) with $\phi=\pi_k$ we see that $\pi_k(T^nx)\to x_0$ as $n\to\infty$ for all $k\in\ZZ_+$ and all $x=(x_k)_{k\ge0}\in c$, uniformly for $x$ in any bounded subset of $c$.  This in particular gives another way of proving that condition~(i) in Theorem~\ref{thm:conv} follows from either (ii) or (iii).
\end{enumerate}
\end{rems}

\section{Rates of convergence}\label{sec:rates}

In this section we strengthen the implication (i)$\implies$(iii) in Theorem~\ref{thm:conv} by showing that for sequences $x\in c$ that satisfy the condition in (i) as well as an additional assumption we obtain a \emph{rate} of convergence in (iii). A quantified result of this type was first obtained in~\cite{AdeLek10}. We state their result in the following form.

\begin{thm}[{\cite[Thm.~1.3]{AdeLek10}}]\label{thm:AdeLek}
Let  $T\colon c\to c$ be the Cesàro operator, and let $x=(x_k)_{k\ge0}\in c$ be such that $x_0=\lim_{k\to\infty}x_k$ and 
\begin{equation}\label{eq:int}
\int_0^\infty e^{-t}\bigg|\sum_{k=1}^\infty (x_k-x_0)\frac{t^{k-1}}{k!}\bigg|\,\dd t<\infty.
\end{equation}
Then $\|T^nx-Px\|_\infty=O(n^{-1/2})$ as  $n\to\infty.$
\end{thm}

In fact, the condition in~\eqref{eq:int} is replaced in~\cite[Thm.~1.3]{AdeLek10} by a slightly more complicated integrability condition, from which~\eqref{eq:int} follows immediately. Furthermore, \cite[Thm.~1.3]{AdeLek10} includes \emph{explicit} constants in the quantified decay statement. The proof of Theorem~\ref{thm:AdeLek} given in~\cite{AdeLek10} takes as its starting point a probabilistic interpretation of powers of the Cesàro operator; the decay rate with explicit constants then comes out of a quantified version of the central limit theorem known as the Berry--Esseen theorem.

We shall give a simpler proof of the convergence rate obtained in Theorem~\ref{thm:AdeLek} by combining the representation of powers of the Cesàro operator as moment-difference operators with a property of generalised Laguerre polynomials (Lemma~\ref{lem:Laguerre} below) and results from operator theory. Our result applies  to a class of sequences that is characterised by a simpler  and potentially weaker  condition than~\eqref{eq:int}. It may be stated as follows.

\begin{thm}\label{thm:rate}
Let  $T\colon c\to c$ be the Cesàro operator, and let $x=(x_k)_{k\ge0}\in c$ be such that $x_0=\lim_{k\to\infty}x_k$ and the series 
\begin{equation}\label{eq:series_cond}
\sum_{k=1}^\infty\frac{x_k-x_0}{k}
\end{equation} 
converges. Then $\|T^nx-Px\|_\infty=O(n^{-1/2})$ as  $n\to\infty.$
\end{thm}

Before proving Theorem~\ref{thm:rate}, we show that the series in~\eqref{eq:series_cond} is convergent whenever~\eqref{eq:int} holds, ensuring that Theorem~\ref{thm:rate} is at least as strong as Theorem~\ref{thm:AdeLek}. This follows by taking $a_k=(x_{k+1}-x_0)/(k+1)$ for $k\ge0$ in the following result. As mentioned in Remark~\ref{rem:Hardy}(b) below, this is a special case of a much stronger result due to Hardy~\cite{Har49}.

\begin{lem}\label{lem:series}
Let $(a_k)_{k\ge0}$ be a sequence of complex numbers such that $a_k=O(k^{-1})$ as $k\to\infty$, and define $f\colon[0,\infty)\to\CC$ by 
$$f(t)= e^{-t}\sum_{k=0}^\infty a_k\frac{t^{k}}{k!},\qquad t\ge0.$$
If $f\in L^1(0,\infty)$ then the series $\sum_{k=0}^\infty a_k$ is convergent and 
\begin{equation}\label{eq:series}
\sum_{k=0}^\infty a_k=\int_0^\infty\! f(t) \,\dd t.
\end{equation}
\end{lem}

\begin{proof}
For $r>1$ we have
$$\sum_{k=0}^\infty \frac{a_k}{r^{k+1}}=\sum_{k=0}^\infty \frac{a_k}{k!}\int_0^\infty e^{-rt}t^{k} \,\dd t.$$
Let $K=\sup_{k\ge0}(k+1)|a_k|$. Then $K<\infty$ by assumption and 
$$\int_0^\infty e^{-rt}\sum_{k=0}^\infty|a_k|\frac{t^{k}}{k!}\,\dd t\le K\int_0^\infty e^{-rt}\,\frac{e^{t}-1}{t}\,\dd t<\infty,\qquad r>1,$$ 
so by the theorems of Tonelli and Fubini we may interchange the order of summation and integration to obtain
$$\sum_{k=0}^\infty \frac{a_k}{r^{k+1}}=\int_0^\infty e^{-rt} \sum_{k=0}^\infty a_k \frac{t^{k}}{k!} \,\dd t,\qquad r>1.$$
Let 
$$f_r(t)=e^{-rt} \sum_{k=0}^\infty a_k \frac{t^{k}}{k!},\qquad r>1,\ t\ge0.$$
Then $|f_r(t)|\le|f(t)|$ for all $r>1$,  $t\ge0,$ and $f_r(t)\to f(t)$ as $r\to1+$ for all $t\ge0$, so  
$$\lim_{r\to1+}\sum_{k=0}^\infty \frac{a_k}{r^{k+1}}=\int_0^\infty\! f(t) \,\dd t$$
by the dominated convergence theorem. Now Littlewood's theorem~\cite{Lit11} yields the result.
\end{proof}

\begin{rems}\label{rem:Hardy}
\begin{enumerate}[(a)]
\item For the purposes of deducing convergence of the series in~\eqref{eq:series_cond} from~\eqref{eq:int} it would suffice to have Lemma~\ref{lem:series} only for $a_k=o(k^{-1})$ as $k\to\infty$. In this case the application of Littlewood's theorem could be replaced by an application of Tauber's  theorem~\cite{Tau97}. 
\item The conclusion of Lemma~\ref{lem:series} remains true under the much weaker conditions that the integral in~\eqref{eq:series} exists as an improper Riemann integral and $a_k=O(k^{-1/2})$ as $k\to\infty$; see~\cite[Thm.~126 \& Thm.~156]{Har49}. The proof of this  result is considerably more involved than that of Lemma~\ref{lem:series}.
\item It follows, conversely, from~\cite[Thm.~122]{Har49} that the integral in~\eqref{eq:series} exists as an improper Riemann integral whenever the series $\sum_{k=0}^\infty a_k$ is convergent, and in this case, too,~\eqref{eq:series} holds. There are large classes of sequences $(a_k)_{k\ge0}$ for which it is even possible to show that $f\in L^1(0,\infty)$, and in particular this holds whenever $(a_k)_{k\ge0}\in\ell^1$. However, we do not know whether this converse of Lemma~\ref{lem:series} holds in general, even under the stronger assumption that  $a_k=o(k^{-1})$ as $k\to\infty$. In particular, we leave open whether requiring convergence of the series in~\eqref{eq:series_cond} is \emph{strictly} weaker than~\eqref{eq:int}. 
\end{enumerate}
\end{rems}

 We continue with an elementary yet important observation.

\begin{lem}\label{lem:range}
 Let $T\colon c\to c$ be the Cesàro operator, and let $x=(x_k)_{k\ge0}\in c$. Then $x\in\Ran(I-T)$ if and only if $x_0=\lim_{k\to\infty}x_k=0$ and the series $\sum_{k=1}^\infty \frac{x_k}{k}$ converges.
\end{lem}

\begin{proof}
Certainly $\Ran(I-T)\subseteq (\Fix T^*)_\circ=\Ker \pi_0\cap\Ker\pi_\infty,$ so if $x=(x_k)_{k\ge0}\in\Ran(I-T)$ then $x_0=\lim_{k\to\infty}x_k=0$.  Furthermore, if $x=y-Ty$ for some $y=(y_k)_{k\ge0}\in c$, then    
\begin{equation}\label{eq:yn}
y_k=y_0+\left(1+\frac1k\right)x_k+\sum_{j=1}^{k-1}\frac{x_j}{j},\qquad k\ge1,
\end{equation}
so the series $\sum_{k=1}^\infty \frac{x_k}{k}$ converges. Conversely, if $x=(x_k)_{k\ge0}\in c$ is such that $x_0=\lim_{k\to\infty}x_k=0$ and the series $\sum_{k=1}^\infty \frac{x_k}{k}$ converges, we may construct a sequence $y=(y_k)_{k\ge0}$ by choosing $y_0\in\CC$ arbitrarily and defining $y_k$ as in \eqref{eq:yn} for $k\ge1$. Then $y\in c$ and $y-Ty=x$, so $x\in\Ran(I-T)$.
\end{proof}

\begin{rem}\label{rem:range}
It is possible to iterate the result in Lemma~\ref{lem:range} and obtain descriptions of the higher-order ranges $\Ran(I-T)^n$ for $n>1$. For instance, a sequence $x=(x_k)_{k\ge0}\in c$ lies in $\Ran(I-T)^2$ if and only if $x_0=\lim_{k\to\infty}x_k=\sum_{k=1}^\infty\frac{x_k}{k}=0$ and  $\lim_{n\to\infty}{\sum_{k=1}^n\log(\frac nk)\frac{x_k}{k}}$ exists.
\end{rem}

Our proof of Theorem~\ref{thm:rate} is based on the following result, which quantifies the rate of decay in Theorem~\ref{thm:KT} for a certain class of operators. The implication (i)$\implies$(ii), which will be crucial for us, was proved in~\cite[Lem.~2.1]{FogWei73} and~\cite[Thm.~4.5.3]{Nev93}, while the converse was later proved in~\cite[Thm~1.2]{Dun08}.

\begin{thm}\label{thm:Dun}
Let $X$ be a complex Banach space, and let $T\colon X\to X$ be a bounded linear operator. The following are equivalent:
\begin{enumerate}[(i)]
\item[\textup{(i)}] There exists $\alpha\in(0,1)$ such that the operator $(1-\alpha)\inv(T-\alpha I)$ is power-bounded;
\item[\textup{(ii)}] $T$ is power-bounded and $\|T^n(I-T)\|=O(n^{-1/2})$ as $n\to\infty$.
\end{enumerate}
\end{thm}

We shall prove Theorem~\ref{thm:rate} by establishing condition~(i) of Theorem~\ref{thm:Dun} when $T$ is the Cesàro operator. In order to do so we require an auxiliary result on the \emph{generalised Laguerre polynomials} $\smash{L_n^{(1)}}$ defined for $n\ge0$ by
$$\L(t)=\sum_{k=0}^n\binom{n+1}{k+1}\frac{(-1)^k}{k!}t^k.$$
The following fact may well be known in the right circles but, since we have been unable to locate a reference, we include a  proof.

\begin{lem}\label{lem:Laguerre}
For $\alpha\in(0,1/2)$ we have
$$\int_0^\infty e^{-\ga t}|\L(t)|\,\dd t\sim\left(\frac{1-\ga}{\ga}\right)^{n+1},\qquad n\to\infty.$$
\end{lem}

\begin{proof}
By~\cite[Thm.~6.32]{Sze75} the roots of $\L$ are contained in the interval $(0,4(n+1))$, so $|\L(t)|=(-1)^n\L(t)$ for $t\ge 4(n+1)$. For $\alpha\in(0,1/2)$, it follows that
$$\begin{aligned}
\int_0^\infty e^{-\ga t}\big||\L(t)|-&(-1)^n\L(t)\big|\,\dd t\\&=\int_0^{4(n+1)} e^{(1/2-\ga) t}e^{-t/2}\big||\L(t)|-(-1)^n\L(t)\big|\,\dd t\\
&\le 2e^{(2-4\ga)(n+1)}\int_0^{4(n+1)} e^{-t/2}|\L(t)|\,\dd t.
\end{aligned}
$$
Furthermore, the estimates in~\cite[p.~699]{AskWai65} imply that 
$$\int_0^{4(n+1)} e^{-t/2}|\L(t)|\,\dd t=O(n^{1/2}),\qquad n\to\infty.$$
Note that $\frac{\ga}{1-\ga}e^{2-4\ga}<1$ for $\alpha\in(0,1/2)$, and hence 
$$\left(\frac{\ga}{1-\ga}\right)^{n+1}\int_0^\infty e^{-\ga t}\big||\L(t)|-(-1)^n\L(t)\big|\,\dd t\to0,\qquad n\to\infty.$$
We also have
$$\begin{aligned}
\int_0^\infty e^{-\ga t}(-1)^n\L(t)\,\dd t&=\sum_{k=0}^n\binom{n+1}{k+1}\frac{(-1)^{n+k}}{k!}\int_0^\infty e^{-\ga t}t^k\,\dd t\\
&=\sum_{k=0}^n\binom{n+1}{k+1}\frac{(-1)^{n+k}}{\alpha^{k+1}}\\
&=\left(\frac{1-\ga}{\ga}\right)^{n+1}+(-1)^n\sim\left(\frac{1-\ga}{\ga}\right)^{n+1}
\end{aligned}$$
as $n\to\infty$, so the result  follows.
\end{proof}

The proof of Theorem~\ref{thm:rate} is now straightforward.

\begin{proof}[Proof of Theorem~\ref{thm:rate}]
Let $x=(x_k)_{k\ge0}\in c$ be such that $x_0=\lim_{k\to\infty}x_k$ and  the series in~\eqref{eq:series_cond} is convergent. Then  $x-Px\in\Ran(I-T)$ by Lemma~\ref{lem:range}. Since $T^nx-Px=T^n(x-Px)$ for all $n\ge0$, the result will follow if we can show that $\|T^n(I-T)\|=O(n^{-1/2})$ as $n\to\infty$. By Theorem~\ref{thm:Dun} it suffices to prove that the operator $T_\alpha=(1-\ga)\inv(T-\ga I)$ is power-bounded for some $\ga\in(0,1)$. We show that $T_\ga$ is power-bounded for all $\ga\in(0,1/2)$. For $x\in c$, $n\ge1$ and  $k\ge0$, it follows from \cite[\S11.12]{Har49} that
$$\pi_k(T^n x)=\frac{1}{(n-1)!}\int_0^1\log(s\inv)^{n-1}G_x^{(k)}(s)\,\dd s,$$
where
$$G_x^{(k)}(s)=\sum_{j=0}^k\binom{k}{j}x_js^j(1-s)^{k-j},\qquad 0\le s\le 1.$$
 Consequently,
$$\begin{aligned}
\pi_k(T_\alpha^n x)&=\frac{1}{(1-\ga)^n}\sum_{j=0}^n\binom{n}{j}(-\ga)^{n-j}\pi_k(T^j x)
\\&=\left(\frac{-\ga}{1-\ga}\right)^n\bigg(x_k+\int_0^1\sum_{j=1}^n\binom{n}{j}\frac{(-1)^j}{(j-1)!}\frac{\log(s\inv)^{j-1}}{\ga^j}G_x^{(k)}(s)\,\dd s\bigg)\\
&=\left(\frac{-\ga}{1-\ga}\right)^n\left(x_k-\frac1\ga\int_0^1 L_{n-1}^{(1)}(\ga\inv\log(s\inv))G_x^{(k)}(s)\,\dd s\right)
\\&=\left(\frac{-\ga}{1-\ga}\right)^n\left(x_k-\int_0^\infty e^{-\ga t} L_{n-1}^{(1)}(t)G_x^{(k)}(e^{-\alpha t})\,\dd t\right)
\end{aligned}$$
for all $k\ge0$, $n\ge1$, $x\in c$ and $\ga\in(0,1)$. Here we used the substitution $t=\ga\inv\log(s\inv)$ in the last step. Since $|\smash{G_x^{(k)}(s)}|\le\|x\|_\infty$ for all $s\in[0,1]$, $k\ge0$ and $x\in c$, it follows that
$$\|T_\alpha^n \|\le \left(\frac{\ga}{1-\ga}\right)^n\left(1+\int_0^\infty e^{-\ga t}|L_{n-1}^{(1)}(t)|\,\dd t\right)$$
for all $n\ge1$ and $\alpha\in(0,1)$. Now by Lemma~\ref{lem:Laguerre} the operator $T_\ga$ is power-bounded for $\ga\in(0,1/2)$, so the result follows.
\end{proof}

\begin{rems}
\begin{enumerate}[(a)]
\item It is possible to obtain a slightly slower rate of convergence by means of a different argument which avoids the use of generalised Laguerre polynomials. Indeed, using the formulas for the resolvent of the Cesàro operator $T\colon c\to c$ given for instance in~\cite{Rea85, Wen75} it is relatively straightforward to show that $\|(e^{i\theta}I-T)\inv\|=O(|\theta|^{-2})$ as $|\theta|\to0$. It then follows from \cite[Thm.~2.5]{Sei15} or \cite[Thm.~2.11]{Sei16} that
\begin{equation*}\label{eq:rate}
\|T^n(I-T)\|=O\bigg(\frac{\log(n)^{1/2}}{n^{1/2}}\bigg),\qquad n\to\infty,
\end{equation*}
which reproduces the rate in Theorem~\ref{thm:rate} up to a logarithmic factor.
\item There is a natural analogue of Theorem~\ref{thm:rate} for the Cesàro operator $T$ defined, by the same formula as before, on the Banach space $\ell^\infty$ of bounded sequences, equipped with the supremum norm. Indeed, if $x=(x_k)_{k\ge0}\in\ell^\infty$ satisfies  
$$\sup_{n\ge1}\bigg|\sum_{k=1}^n\frac{x_k-x_0}{k}\bigg|<\infty,$$ 
then $\|T^nx-Px\|_\infty=O(n^{-1/2})$ as  $n\to\infty$, where $Px=x_0 e_\infty$ as before. The proof is essentially the same as that of Theorem~\ref{thm:rate}, requiring only a straightforward modification of Lemma~\ref{lem:range}. 
\end{enumerate}
\end{rems}

Let us write
$$X=\big\{x\in c:(T^nx)_{n\ge0} \mbox{ converges in norm as $n\to\infty$}\big\},$$
as in the proof of Theorem~\ref{thm:conv}, and recall that a sequence $x=(x_k)_{k\ge0}\in c$ lies in $X$ if and only if $x_0=\lim_{k\to\infty}x_k$.
Our next result shows that the rate of convergence obtained in Theorem~\ref{thm:rate} cannot be improved, and that there cannot be a rate that holds for \emph{all} sequences $x\in X$. Similar statements would follow from the lower bounds mentioned in the discussion after~\cite[Thm.~1.3]{AdeLek10}.

\begin{prp}\label{prp:opt}
Let $T\colon c\to c$ be the Cesàro operator.
\begin{enumerate}[(a)]
\item[\textup{(a)}] Suppose that $(r_n)_{n\ge0}$ is a sequence of positive numbers such that $r_n=o(n^{-1/2})$ as $n\to\infty$. Then there exists $x\in X$ such that the series in~\eqref{eq:series_cond} is convergent but 
$$\limsup_{n\to\infty}\,r_n^{-1}\|T^n x-Px\|_\infty=\infty.$$
\item[\textup{(b)}] Suppose that $(r_n)_{n\ge0}$ is a sequence of positive numbers such that $r_n\to0$ as $n\to\infty$. Then there exists $x\in X$ such that $\|T^nx-Px\|_\infty\ge r_n$ for all $n\ge0$.
\end{enumerate}
\end{prp}

\begin{proof}
In order to prove part~(a) suppose, to the contrary, that $\|T^n x-Px\|=O(r_n)$ as $n\to\infty$ for all $x\in X$ such that series in~\eqref{eq:series_cond} is convergent. By Lemma~\ref{lem:range} it follows that $\|T^n x\|=O(r_n)$ as $n\to\infty$ for all $x\in\Ran(I-T)$, and hence $\|T^n(I-T)\|=O(r_n)$ as $n\to\infty$ by an application of the uniform boundedness principle. Now since $\sigma(T)=\{z\in\CC:|z-\frac12|\le\frac12\}$ we see that
$$\|(e^{i\theta}-T)\inv\|\ge\frac{1}{\dist(e^{i\theta},\sigma(T))}\ge\frac{1}{|e^{i\theta}-\frac12(1+e^{2i\theta})|}=\frac{1}{1-\cos\theta}\ge\frac{2}{|\theta|^2}$$
for $0<|\theta|\le\pi$, and hence $\limsup_{n\to\infty}n^{1/2}\|T^n(I-T)\|>0$ by~\cite[Prop.~2.1]{NgSei20}. By assumption we have $r_n=o(n^{-1/2})$ as $n\to\infty$, giving the required contradiction.

In order to prove part~(b) we note first that if $T_0$ denotes the restriction of $T$ to $Y$, the closure of $\Ran(I-T)$, then the spectral radius $r(T_0)$ of $T_0$ satisfies $r(T_0)=1$. Indeed, we certainly have $r(T_0)\le\|T\|=1$. On the other hand,  if we had $r(T_0)<1$ then we would obtain $\|T^n(I-T)\|=O(\alpha^n)$ as $n\to\infty$ for some $\alpha\in[0,1)$, contradicting the fact that $\limsup_{n\to\infty}n^{1/2}\|T^n(I-T)\|>0$, as established in the proof of part~(a). It follows from~\cite[Thm.~1]{Mue88} that there exists $x\in Y\subseteq X$ such that $\|T_0^nx\|\ge r_n$ for all $n\ge0$. Since $Y\subseteq \Ker P$ we have  $T_0^nx=T^nx-Px$ for all $n\ge0$, and the result follows.
\end{proof}

\begin{rem}
In fact, \cite[Thm.~1]{Mue88} tells us that, given a sequence $(r_n)_{n\ge0}$ of positive numbers such that $r_n\to0$ as $n\to\infty$ and given any $\ep>0$,  the sequence $x\in c$ in part~(b) may be chosen in  such a way that $\|x\|_\infty<\sup_{n\ge0}r_n+\ep$. Furthermore, using~\cite[Thm.~37.8(ii)]{Mue07} we see that there is a dense subset $Z$ of $X$ such that for every $x\in Z$ we have $\|T^nx-Px\|_\infty\ge r_n$ for all \emph{sufficiently large} $n\ge0$.
\end{rem}

Even though part~(a) of Proposition~\ref{prp:opt} shows that the rate of convergence in Theorem~\ref{thm:rate} cannot be improved for \emph{all} sequences, it is clear from the proof of the latter result that for certain sequences one \emph{can} do better. In particular, if $x\in \Ran(I-T)^k\oplus\langle e_\infty\rangle$ for some $k\ge1$ then, using the fact that $\|T^n(I-T)\|=O(n^{-1/2})$ as $n\to\infty$, we see that $\|T^nx-Px\|=O(n^{-k/2})$ as $n\to\infty$. As noted in Remark~\ref{rem:range}, we may in principle determine whether a sequence $x\in c$ lies in $\Ran(I-T)^k$ for a given $k\ge1$ by iterating the description of $\Ran(I-T)$ given in Lemma~\ref{lem:range}. We know from Remark~\ref{rem1}(a) that $\Ran(I-T)\oplus\langle e_\infty\rangle$ is a dense subspace of $X$. It follows that $\Ran(I-T)^k\oplus\langle e_\infty\rangle$ is a dense subspace of $X$ for every $k\ge1$, and by part~(b) of Proposition~\ref{prp:opt} it is a proper subspace of $X$. In fact, a slightly stronger statement is true, as we now show.  For $x\in X$ we say that $(T^nx)_{n\ge0}$ converges \emph{rapidly} if $\|T^nx-Px\|_\infty=O(n^{-r})$ as $n\to\infty$ for all $r>0$, and we write $X_\infty$ for the set of all $x\in X$ such that $(T^nx)_{n\ge0}$ converges rapidly. 

\begin{prp}\label{prp:rapid}
 The set $X_\infty$ is a dense  subspace of $X$.
\end{prp}

\begin{proof}
It is easy to see that $X_\infty$ is a subspace of $X$. We prove that $X_\infty$ is dense in $X$. To this end, let $Z=\bigcap_{k=0}^\infty \Ran(I-T)^k\oplus \langle e_\infty\rangle$. As noted above, $\|T^n(I-T)^k\|=O(n^{-k/2})$ as $n\to\infty$ for all $k\ge1$, and hence  $ Z\subseteq X_\infty$. An application of  Esterle's Mittag-Leffler-type  theorem~\cite{Est84} and Remark~\ref{rem1}(a) now shows that $Z$ is dense in $X$, and hence so is $X_\infty$.
\end{proof}

\section{The Cesàro operator on spaces of continuous functions}\label{sec:cont}

In this final section we consider, relatively briefly, two variants of the Cesàro operator which act on spaces of continuous functions. We let $C[0,1]$ denote the space of continuous functions $f\colon[0,1]\to\CC$ and we let $C_\infty[0,\infty)$ denote the space of continuous functions $f\colon[0,\infty)\to\CC$ such that $\lim_{t\to\infty}f(t)$ exists. We endow both spaces with the supremum norm, making them complex Banach spaces, and we define the 
Cesàro operator $T$ on either space by $(Tf)(0)=f(0)$ and 
\begin{equation}\label{eq:Cesaro_cont}
(Tf)(t)=\frac1t\int_0^tf(s)\,\dd s,\qquad t>0.
\end{equation}
In each case   $T$ is a well-defined bounded linear operator of norm  $\|T\|=1$. Motivated by \cite{AdeLek10,GalSol08} we study convergence of orbits $(T^nf)_{n\ge0}$ for functions $f$ lying in either of the spaces $C[0,1]$ or $C_\infty[0,\infty)$. In doing so it will be useful to have  the following facts at our disposal. Note that one of the two implications in part~(a) appears, with the same proof, in~\cite[Prop.~2.4(ii)]{AlBoRi15}.

\begin{lem}\phantomsection\label{lem:range_cont}
\begin{enumerate}[(a)]
\item[\textup{(a)}] Let $T\colon C[0,1]\to C[0,1]$ be the Cesàro operator, and let $f\in C[0,1]$. Then $f\in \Ran(I-T)$ if and only if $f(0)=0$ and the improper Riemann integral
\begin{equation}\label{eq:imp1}
\int_0^1 \frac{f(t)}{t}\,\dd t
\end{equation}
 exists.
\item[\textup{(b)}] Let $T\colon C_\infty[0,\infty)\to C_\infty[0,\infty)$ be the Cesàro operator, and let $f\in C_\infty[0,\infty)$. Then $f\in \Ran(I-T)$ if and only if $f(0)=\lim_{t\to\infty}f(t)=0$ and the improper Riemann integral
\begin{equation}\label{eq:imp2}
\int_0^\infty \frac{f(t)}{t}\,\dd t
\end{equation}
exists.
\end{enumerate}
\end{lem}

\begin{proof}
We prove part (a). Suppose first that $f\in\Ran(I-T)$, so that $f=g-Tg$ for some $g\in C[0,1]$. Since $(Tg)(0)=g(0)$ by definition we see immediately that $f(0)=0$. Moreover, the function $Tg$ is differentiable on $(0,1]$ and its derivative satisfies
$$(Tg)'(t)=\frac{g(t)}{t}-\frac{1}{t^2}\int_0^tg(s)\,\dd s=\frac{f(t)}{t},\qquad 0<t\le 1.$$
Hence  
$$\int_\ep^t\frac{f(s)}{s}\,\dd s=(Tg)(t)-(Tg)(\ep),\qquad 0<\ep\le t\le 1.$$
Since $Tg\in C[0,1]$ the improper Riemann integral in~\eqref{eq:imp1} exists. Conversely, if $f\in C[0,1]$ satisfies $f(0)=0$ and the improper Riemann integral in~\eqref{eq:imp1} exists we may define a function $g\colon[0,1]\to\CC$ by  
\begin{equation}\label{eq:g_def}
g(t)=-\int_t^1\frac{f(s)}{s}\,\dd s,\qquad 0\le t\le 1,
\end{equation}
where the integral is interpreted as an improper Riemann integral when $t=0$.
Then  $g\in C[0,1]$ and we may define $h\in C[0,1]$ by $h=f+g$. The function $g$ is differentiable on $(0,1]$ and satisfies $tg'(t)= f(t)$ for $0<t\le 1$. Hence
$$(Th)(t)=\frac1t\int_0^t\big(f(s)+g(s)\big)\,\dd s=\frac1t\int_0^t\big(sg(s)\big)'\,\dd s=g(t),\qquad 0<t\le 1.$$
Since $f(0)=0$, we also have $(Th)(0)=g(0)$. Thus $Th=g=h-f$ and hence $f=h-Th\in\Ran(I-T)$, as required.

The proof of part~(b) is similar. In particular, we note that if $g\in C_{0,\infty}[0,\infty)$ and if $f=g-Tg$ then $f(0)=\lim_{t\to\infty}f(t)=0$. Existence of the improper Riemann integral in~\eqref{eq:imp2} follows as in part~(a). For the converse implication we simply extend the domain of definition of the function $g$ defined in~\eqref{eq:g_def} to the half-line $[0,\infty)$, so that
$$g(t)=\int_1^t\frac{f(s)}{s}\,\dd s,\qquad t>1.$$
We leave it to the reader to verify the details.
\end{proof}

Our first main result in this section reproduces~\cite[Thm.~3]{GalSol08} concerning the Cesàro operator on $C[0,1]$, and supplements it with a rate of convergence for functions satisfying a certain integrability condition. To our knowledge, the quantified statement has not previously appeared in the literature. 
 Here we let $P\colon C[0,1]\to C[0,1]$ denote the operator defined by $(Pf)(t)=f(0)$ for  $f\in C[0,1]$ and $0\le t\le 1$.

\begin{thm}\label{thm:interval}
Let $T\colon C[0,1]\to C[0,1]$ be the Cesàro operator. Then $\|T^nf-Pf\|_\infty\to0$ as $n\to\infty$ for all $f\in C[0,1]$. Furthermore, if $f\in C[0,1]$ is such that the improper Riemann integral
\begin{equation}\label{eq:imp3}
\int_0^1\frac{f(t)-f(0)}{t}\,\dd t
\end{equation}
exists, then $\|T^nf-Pf\|_\infty=O(n^{-1/2})$ as $n\to\infty$.
\end{thm}

\begin{proof}
 It follows from the fundamental theorem of calculus that $\Fix T$ consists precisely of all constant functions on $[0,1]$. Moreover, a standard application of the Weierstrass approximation theorem shows that the closure $Y$ of $\Ran(I-T)$ is given by $Y=\{f\in C[0,1]:f(0)=0\}$; see~\cite[Prop.~2.4(i)]{AlBoRi15}. Hence $C[0,1]=\Fix T\oplus Y$. Since  $\sigma(T)=\{z\in\CC:|z-\frac12|\le\frac12\}$ by~\cite[Prop.~2.1]{AlBoRi15} it follows from Theorem~\ref{thm:KT}, as in the proof of Theorem~\ref{thm:conv}, that $\|T^nf\|_\infty\to0$ as $n\to\infty$ for all $f\in Y$. Finally, given an arbitrary $f\in C[0,1]$, we have $Pf\in\Fix T$ and $f-Pf\in Y$, and hence $\|T^nf-Pf\|_\infty=\|T^n(f-Pf)\|_\infty\to0$ as $n\to\infty$, as required. 
We now turn to the quantified statement. By~\cite[Lem.~2]{Boy68} we have
$$(T^nf)(t)=\frac{1}{(n-1)!}\int_0^1\log(s\inv)^{n-1}f(st)\,\dd s$$
for all $f\in C[0,1]$,  $t\in[0,1]$ and $n\ge1$. Given $\alpha\in(0,1)$ let $T_\alpha=(1-\ga)\inv(T-\ga I)$. A calculation similar to that in the proof of Theorem~\ref{thm:rate} shows that 
$$
(T_\alpha^nf)(t)=\left(\frac{-\alpha}{1-\alpha}\right)^n\bigg(f(t)-\int_0^\infty e^{-\alpha s}L_{n-1}^{(1)}(s)f(te^{-\alpha s})\,\dd s\bigg)
$$
for all $f\in C[0,1]$,  $t\in[0,1]$, $n\ge1$ and $\ga\in(0,1)$, and hence
$$\|T_\alpha^n\|\le  \left(\frac{\ga}{1-\ga}\right)^n\left(1+\int_0^\infty e^{-\ga s}|L_{n-1}^{(1)}(s)|\,\dd s\right)$$
for all $n\ge1$ and $\ga\in(0,1)$.   By Lemma~\ref{lem:Laguerre} the operator $T_\alpha$ is power-bounded for $\alpha\in(0,1/2)$, and hence Theorem~\ref{thm:Dun} gives $\|T^n(I-T)\|=O(n^{-1/2})$ as $n\to\infty$. Thus if $f\in C[0,1]$ is such that the improper Riemann integral in~\eqref{eq:imp3} exists, then $f-Pf\in\Ran(I-T)$ by  Lemma~\ref{lem:range_cont}(a), so $\|T^nf-Pf\|_\infty=\|T^n(f-Pf)\|_\infty=O(n^{-1/2})$ as $n\to\infty$, as required.
\end{proof}

\begin{rems}\phantomsection \label{rem:int}
\begin{enumerate}[(a)]
\item In the first part of the proof establishing the unquantified statement we may alternatively deduce the result directly from the Weierstrass approximation theorem and a standard approximation argument as in~\cite[Thm.~3]{GalSol08}, without appealing to Theorem~\ref{thm:KT}.
\item Appropriate analogues of Propositions~\ref{prp:opt} and~\ref{prp:rapid} hold also for the Cesàro operator on $C[0,1]$, with near-identical proofs.
\item Let $B_0[0,1]$ denote the complex vector space of bounded measurable functions $f\colon[0,1]\to\CC$ which are continuous at zero. Then $B_0[0,1]$ is a Banach space when endowed with the supremum norm, and the Cesàro operator extends to a well-defined bounded linear operator $T\colon B_0[0,1]\to B_0[0,1]$. Furthermore, the following variant of the quantified statement in Theorem~\ref{thm:interval} holds: If $f\in B_0[0,1]$ is such that the  integral in~\eqref{eq:imp3} exists as an improper \emph{Lebesgue} integral then $\|T^nf-Pf\|_\infty=O(n^{-1/2})$ as $n\to\infty$, where $(Pf)(t)=f(0)$ for all $t\in[0,1]$ as before.
The proof requires only a straightforward modification of Lemma~\ref{lem:range_cont}(a) based on the fundamental theorem of calculus for the Lebesgue integral; see for instance~\cite[Thm.~3.11]{SteSha05}.
\item The same argument as in the proof of Theorem~\ref{thm:interval} shows that $\|T^n(I-T)\|=O(n^{-1/2})$ as $n\to\infty$ when $T$ is the Cesàro operator on $L^\infty(0,1)$, defined as in~\eqref{eq:Cesaro_cont} for all $f\in L^\infty(0,1)$ and almost all $t\in(0,1)$. By another simple modification of Lemma~\ref{lem:range_cont}(a) we then find that if $f\in L^\infty(0,1)$ and there exists a constant $\lambda\in\CC$ such that 
\begin{equation}\label{eq:sup_int}
\sup_{0<\ep\le1}\bigg|\int_\ep^1\frac{f(t)-\lambda}{t}\,\dd t\bigg|<\infty,
\end{equation}
then 
$\|T^nf-g\|_\infty=O(n^{-1/2})$ as $n\to\infty$, where $g\in L^\infty(0,\infty)$ is the function defined by $g(t)=\lambda$ for almost all $t\in(0,1)$.
\end{enumerate}
\end{rems}

Finally, we consider the Cesàro operator on $C_\infty[0,\infty)$ and describe the asymptotic behaviour of its orbits. As in the previous case, we define $P\colon C_\infty[0,\infty)\to C_\infty[0,\infty)$ by $(Pf)(t)=f(0)$ for all $f\in C_\infty[0,\infty)$ and $t\ge0$.  Note that
 the equivalence between (i) and (iii) in part~(a) was first proved, by a different method, in~\cite[Thm.~4]{GalSol08}; the relationship between the quantified statement in part~(b) and the results obtained in~\cite{AdeLek10} is discussed in the subsequent remarks.

\begin{thm}\label{thm:half}
Let $T\colon C_\infty[0,\infty)\to C_\infty[0,\infty)$ be the Cesàro operator. 
\begin{enumerate}[(a)]
\item[\textup{(a)}]
Let $f\in C_\infty[0,\infty)$. The following are equivalent:
\begin{enumerate}[(i)]
\item[\textup{(i)}] $f(0)=\lim_{t\to\infty}f(t)$;
\item[\textup{(ii)}] $(T^nf)_{n\ge0}$ converges weakly as $n\to\infty$;
\item[\textup{(iii)}] $(T^nf)_{n\ge0}$ converges in norm as $n\to\infty$.
\end{enumerate}
Furthermore, if these conditions hold then 
\begin{equation}\label{eq:limit_cont}
\|T^nf-Pf\|_\infty\to0,\qquad n\to\infty.
\end{equation}

\item[\textup{(b)}] If $f\in C[0,1]$ is such that $f(0)=\lim_{t\to\infty}f(t)$ and the improper Riemann integral
\begin{equation}\label{eq:imp4}
\int_0^\infty\frac{f(t)-f(0)}{t}\,\dd t
\end{equation}
exists, then $\|T^nf-Pf\|_\infty=O(n^{-1/2})$ as $n\to\infty$.
\end{enumerate}
\end{thm}

\begin{proof}
In order to prove part~(a) let $f\in C_\infty[0,\infty)$. It is clear that (iii)$\implies$(ii), so suppose that~(ii) holds and let $g\in  C_\infty[0,\infty)$ be such that $T^nf\to g$ weakly as $n\to\infty$. Then $g\in\Fix T$ and an application of the fundamental theorem of calculus shows that $g$ must be constant.  Define the functionals $\pi_0,\pi_\infty\in C_\infty[0,\infty)^*$ by $\pi_0(f)=f(0)$ and $\pi_\infty(f)=\lim_{t\to\infty}f(t)$. Then $\pi_0,\pi_\infty\in\Fix T^*$,  and it follows from weak convergence of the sequence $(T^nf)_{n\ge0}$ that $\pi_0(f)=\pi_0(g)$ and $\pi_\infty(f)=\pi_\infty(g)$. Since $g$ is constant we deduce that $f(0)=\lim_{t\to\infty}f(t)$, as required. The argument also shows that $g=Pf$.  Now suppose that~(i) holds. Then $f-Pf$ lies in the closure of $\Ran(I-T)$ by~\cite[Thm.~2.6]{AlBoRi15}. Furthermore, by~\cite[Prop.~2.2]{AlBoRi15} the spectrum of $T$ satisfies $\sigma(T)=\{z\in\CC:|z-\frac12|=\frac12\}$. It follows from Theorem~\ref{thm:KT}, as in the proofs of Theorems~\ref{thm:conv} and~\ref{thm:interval}, that $\|T^nf-Pf\|_\infty=\|T^n(f-Pf)\|_\infty\to0$ as $n\to\infty$, establishing that (i)$\implies$(iii). The argument also shows that the limit, if it exists, is of the form given in~\eqref{eq:limit_cont}.

We now turn to part~(b). If $f\in C_\infty[0,\infty)$ is such that $f(0)=\lim_{t\to\infty}f(t)$ and the improper Riemann integral in~\eqref{eq:imp4} exists then $f-Pf\in\Ran(I-T)$ by part~(b) of Lemma~\ref{lem:range_cont}. Arguing exactly as in Theorem~\ref{thm:interval} we may use Theorem~\ref{thm:Dun} to show that $\|T^n(I-T)\|=O(n^{-1/2})$ as $n\to\infty$, and the result now follows as in the proofs of Theorems~\ref{thm:rate} and~\ref{thm:interval}.
\end{proof}

\begin{rems}
\begin{enumerate}[(a)]
\item As in the case of the Cesàro operator on $C[0,1]$, appropriate analogues of Propositions~\ref{prp:opt} and~\ref{prp:rapid} hold also for the Cesàro operator on $C_\infty[0,\infty)$.
\item Let $B_{0,\infty}[0,\infty)$ denote the complex vector space of bounded measurable functions $f\colon[0,\infty)\to\CC$ which are continuous at zero and converge to a limit at infinity. Then $B_{0,\infty}[0,\infty)$ is a Banach space when endowed with the supremum norm, and the Cesàro operator extends to a well-defined bounded linear operator $T\colon B_{0,\infty}[0,\infty)\to B_{0,\infty}[0,\infty)$. The following variant of Theorem~\ref{thm:half}(b) holds in this case: If $f\in B_{0,\infty}[0,\infty)$ is such that $f(0)=\lim_{t\to\infty}f(t)$ and the integral in~\eqref{eq:imp4} exists as an improper Lebesgue integral, then $\|T^nf-Pf\|_\infty=O(n^{-1/2})$ as $n\to\infty$, where $(Pf)(t)=f(0)$ for all $t\ge0$. As before, the proof requires only a straightforward modification of Lemma~\ref{lem:range_cont}(b).
\item We may compare this result for the Cesàro operator on $B_{0,\infty}[0,\infty)$ with~\cite[Thm.~1.1]{AdeLek10}, where the same decay rate  is obtained for (real-valued) functions $f\in B_{0,\infty}[0,\infty)$ satisfying $f(0)=\lim_{t\to\infty}f(t)$ and an integrability condition for a certain symmetrised version of the function $f$. This integrability condition immediately implies that 
\begin{equation}\label{eq:integrable}
\int_0^\infty\frac{|f(t)-f(0)|}{t}\,\dd t<\infty.
\end{equation}
Note that even~\eqref{eq:integrable} is strictly stronger than requiring the integral in~\eqref{eq:imp4} to exist as an improper Lebesgue integral. For instance, the function $f\in C_\infty[0,\infty)\subseteq B_{0,\infty}[0,\infty)$ defined by
$$f(t)=\frac{\sin(t)}{\log(2+t)},\qquad t\ge0,$$
satisfies $f(0)=\lim_{t\to\infty}f(t)=0$. Moreover, the integral in~\eqref{eq:imp4} exists as an improper Riemann integral, and hence as an improper Lebesgue integral, but \eqref{eq:integrable} does not hold. It follows that the result in part~(b) above applies to a strictly larger class of functions $f\in B_{0,\infty}[0,\infty)$ than~\cite[Thm.~1.1]{AdeLek10}, and in fact even part~(b) of Theorem~\ref{thm:half} applies to functions $f\in C_\infty[0,\infty)$ which are not covered by~\cite[Thm.~1.1]{AdeLek10}.
\item Analogous comments to those made in Remark~\ref{rem:int}(d) for the Cesàro operator on $L^\infty(0,1)$ apply to the Cesàro operator on $L^\infty(0,\infty)$. In particular, if $f\in L^\infty(0,\infty)$ and there exists $\lambda\in\CC$ such that
$$\sup_{0<\ep\le R}\bigg|\int_\ep^R\frac{f(t)-\lambda}{t}\,\dd t\bigg|<\infty,$$
then 
$\|T^nf-g\|_\infty=O(n^{-1/2})$ as $n\to\infty$, where $g\in L^\infty(0,\infty)$ is defined by $g(t)=\lambda$ for almost all $t>0$.
\end{enumerate}
\end{rems}


\end{document}